\documentclass[11pt]{article}
\usepackage{latexsym,amsmath,color,amsthm,amssymb,epsfig,graphicx,mathrsfs}
\usepackage{amssymb}
\usepackage[left=1in,top=1in,right=1in,bottom=1in]{geometry}
\usepackage{setspace}
\usepackage{amssymb, amsmath, amsthm, graphicx,mathrsfs}

\usepackage{comment}
\def\qed{\hfill\ifhmode\unskip\nobreak\fi\quad\ifmmode\Box\else\hfill$\Box$\fi}
\def\ite#1{\hfill\break${}$\hbox to 50pt {\quad(#1)\hfill}}
\def\Ber{{{\rm B}}}
\def\IB{{\rm B}_{\rm ind}}

\def\ex{{\rm{ex}}}

\def\cA{{\mathcal A}}
\def\cB{{\mathcal B}}

\def\cF{{\mathcal F}}

\def\cH{{\mathcal H}}

\def\cN{{\mathcal N}}
\def\GT{{\mathcal G}_{\rm tri}}  % ??? XXX !!!

\newtheorem{thm}{Theorem}[section]
\newtheorem{cor}[thm]{Corollary}

\newtheorem{definition}[thm]{Definition}

\newtheorem{lem}[thm]{Lemma}

\newtheorem{claim}[thm]{Claim}

\begin{document}

\pagestyle{myheadings}
\markright{{\small{\sc Z.~F\"uredi and Ruth Luo:   Induced Tur\'an problems and traces of hypergraphs  %%%% , \enskip 02/17/2020
}}}

\title{\vspace{-0.5in}  Induced Tur\'an problems and traces of hypergraphs}

\author{
{{Zolt\'an F\" uredi}}\thanks{
\footnotesize {Alfr\' ed R\' enyi Institute of Mathematics, Hungary.
E-mail:  \texttt{z-furedi@illinois.edu}. %% \texttt{furedi.zoltan@renyi.mta.hu}.
Research supported in part by the Hungarian National Research, Development and Innovation Office NKFIH grant KH-130371.
}}
\and{{Ruth Luo}}\thanks{University of California, San Diego, La Jolla, CA 92093, USA. E-mail: {\tt ruluo@ucsd.edu}. Research of this author
is supported in part by NSF grant DMS-1902808.
}}
\date{\today}

\maketitle

\vspace{-0.3in}

\begin{abstract}
Let $F$ be a graph. We say that a hypergraph $\cH$ contains an \emph{induced Berge} $F$ if
the vertices of $F$ can be embedded to $\cH$ (e.g., $V(F)\subseteq V(\cH)$) and
 there exists an injective mapping $f$ from the edges of $F$ to the hyperedges of $\cH$ such that
 $f(xy) \cap V(F) = \{x,y\}$ holds for each edge $xy$ of $F$.
In other words, $\cH$ contains $F$ as a trace.

 Let $\ex_{r}(n,\IB F)$ denote the maximum number of edges in an $r$-uniform hypergraph with no induced Berge $F$. Let $\ex(n,K_r, F)$ denote the maximum number of $K_r$'s in an $F$-free graph on $n$ vertices.
We show that these two Tur\'an type functions are strongly related.

\end{abstract}

\section{ \bf  Definitions, Berge $F$ subhypergraphs}

A hypergraph $\cH$ is $r$-\emph{uniform}  or simply an {\em $r$-graph} if it is a family of $r$-element subsets of a finite set $V(\cH)$.
If the \emph{vertex set} $V(\cH)$ is clear from the text, then we associate an $r$-graph $\cH$ with its edge set $E(\cH)$.
Usually we take $V(\cH)=[n]$, where $[n]$ is the set of first $n$ integers, $[n]:=\{ 1, 2, 3,\dots, n\}$.
We also use the notation $\cH\subseteq \binom{[n]}{r}$. For a set of vertices $S \subseteq V(\cH)$ define the \emph{codegree} of $S$, denoted as $\deg(S)$, to be the number of edges of $\cH$ containing $S$.
The $s$-\emph{shadow}, $\partial_s\cH$, is the family of $s$-sets contained in the edges of $\cH$. So $\partial_1\cH$ is the set of non-isolated vertices, and $\partial_2\cH$ is the graph whose edges are the pairs with positive co-degree in $\cH$.

\begin{definition} %[Anstee and Salazar~\cite{AS}, Gerbner and Palmer~\cite{GP}]
For a graph $F$ with vertex set $\{v_1, \ldots, v_p\}$ and edge set $\{e_1, \ldots, e_q\}$, a hypergraph $\mathcal H$ contains a {\bf Berge $F$} if there exist distinct vertices $\{w_1, \ldots, w_p\} \subseteq V(\mathcal H)$ and  distinct edges $\{f_1, \ldots, f_q\} \subseteq E(\mathcal H)$, such that if $e_i = v_{\alpha} v_{\beta}$, then $\{w_{\alpha}, w_{\beta}\} \subseteq f_i$. The vertices $\{w_1, \ldots, w_p\}$ are called the {\bf base vertices} of the Berge $F$.
\end{definition}

\begin{definition} For a graph $F$ with vertex set $\{v_1, \ldots, v_p\}$ and edge set $\{e_1, \ldots, e_q\}$, a hypergraph $\mathcal H$ contains an {\bf induced Berge $F$} if there exists a set of distinct vertices $W:=\{w_1, \ldots, w_p\} \subseteq V(\mathcal H)$ and distinct edges $\{f_1, \ldots, f_q\} \subseteq E(\mathcal H)$, such that if $e_i = v_{\alpha} v_{\beta}$, then  $\{w_{\alpha}, w_{\beta}\} = f_i\cap W$.
\end{definition}
In particular, in the case that $\mathcal H$ is a graph (2-uniform), an induced Berge $F$ is just any copy of $F$ in $\mathcal H$, not to be confused with the notion of induced subgraphs.
If the two hypergraphs have the same number of edges, $e(\cH)=e(\cF)$,  then we say that $\cH$ itself is a(n induced) Berge $F$ hypergraph. 
The set of $r$-uniform (induced) Berge $F$ hypergraphs is  denoted by $\{\Ber (F)\}_{r}$ ($\{\IB (F)\}_{r}$, resp.).
For example, if $F$ is a triangle, $E(F)= \{ 12, 13, 23\}$, then $\{B(F)\}_3$ contains four triple systems:
$\{12a, 13a, 23a \}$, $\{12a, 13a, 23b \}$, $\{12a, 13b, 23c \}$, and $\{123, 13a, 23b \}$. The first three of them contains an induced $C_3$, the fourth does not.
Parenthesis and indices are omitted when it does not cause ambiguities.

\subsection{Three types of extremal numbers}

Given a set of $r$-graphs $\cF$ the hypergraph $\cH$ is called $\cF$-\emph{free} if it does not have any subgraph isomorphic to any member of $\cF$.
The \emph{Tur\'an number} of $\cF$, denoted by $\ex_{r}(n, \cF)$, is the maximum size of an $\cF$-free $\cH\subseteq \binom{[n]}{r}$.
Usually it is assumed that $|\cF|$ is finite, so the  well-known fact $\ex_{2}(n, \{ C_3, C_4, C_5, \dots\})= n-1$ usually is not considered a Tur\'an type result because the set of forbidden graphs $\cF$, the set of all cycles, is infinite.
If $r=2$  then the index is usually omitted.
Also if $\cF$ has only one member, $\cF=\{ F\}$, then we write $\ex_{r}(n,F)$ instead of $\ex_{r}(n, \{ F\})$.

The {\em generalized Tur\'an number} for graphs, pioneered by %%% P.~
 Erd\H os~\cite{Erdos} and recently systematically investigated by Alon and Shikhelman~\cite{alon-shik}, is the following extremal problem.
We only formulate the case relevant to this paper.
Given a graph $F$, let $\ex(n,K_r,F)$ denote the maximum possible number of copies of $K_r$'s in an $F$-free, $n$-vertex graph, i.e.,
\[
\ex(n,K_r, F) := \max \left\{  |\cN_r(H)|: H \text{ is }F\text{-free }, H \subseteq {[n] \choose 2} \right\},
\]
where $\cN_r(H)\subseteq \binom{[n]}{r}$ is the family of $r$-element vertex sets that span a $K_r$ in $H$.
In particular $\cN_2(H)= E(H)$ and $\ex(n, K_2, F) = \ex(n, F)$ is the regular Tur\'an number of $F$.

For a graph $F$ and positive integer $r$, let
\[\ex_{r}(n, \Ber F) := \max \{e(\mathcal H): \mathcal H \subseteq {[n] \choose r} \text{ and } \mathcal H \text{ is Berge } F\text{-free}\}.
 \] 
Ever since Gy\H ori, G.~Y.~Katona, and Lemons~\cite{GKL} investigated hypergraphs without long Berge paths there is a  renewed interest concerning extremal Berge type problems.
Here we define a related function, the \emph{induced Berge Tur\'an number} of $F$.
Special cases were studied earlier, especially the 3-uniform case (e.g., Maherani and Shahsiah~\cite{MS3}, Gy\'arf\'as~\cite{Gy}, Sali and Spiro~\cite{SS}).
\[ \ex_{r}(n,\IB F) := \max \{e(\mathcal H): \mathcal H \subseteq {[n] \choose r} \text{ and } \mathcal H \text{ is induced Berge } F\text{-free}\}.\]

We consider the relationship between these three functions.
Obviously,
\begin{equation}\label{eqindF}
 \ex(n,K_r, F) \leq \ex_{r}(n,\Ber F) \leq \ex_{r}(n, \IB F). 
\end{equation}
Indeed, consider a graph $G$ with $|\cN_r(G)|=\ex(n,K_r, F)$.
Since $G$ is $F$-free,  the $r$-graph $\cN_r(G)$ is Berge $F$-free, implying $|\cN_r(G)|\leq \ex_{r}(n, \Ber F)$.
The second inequality holds because if a hypergraph contains no Berge $F$ then it also contains no induced Berge $F$.

The induced Berge $F$ problem is motivated by the forbidden configuration problem for matrices (see Anstee~\cite{ans} for a survey).
It can also be reformulated as a hypergraph trace problem (see, e.g., Mubayi and Zhao~\cite{MZ}).
Few results are known for the induced Berge Tur\'an problem. In~\cite{MZ}, the value of $\ex_r(n, \IB K_t)$ is determined asymptotically for $K_3$ and $K_4$, as well as $K_t$ when $t$ is close to the uniformity $r$.

A special case of induced Berge hypergraphs, so called \emph{expansions} were intensively studied,
see, e.g., Pikhurko~\cite{P},  Kostochka, Mubayi, and Verstra\"ete~\cite{KMV}, and the survey by Mubayi and Verstra\"ete~\cite{MV}.

There are also other areas of research in extremal graph theory which are called `induced' Tur\'an type results. E.g.,
Pr\"omel and Steger~\cite{PS} investigated the extremal properties of graphs not containing an induced copy of a given graph $F$. A more recent version is by Loh, Tait, Timmons, and Zhou~\cite{LTTZ}. But most of these are only distant relatives of our induced Berge question.

\section{Main results, bounds for $\ex_r(n, \IB F)$}

\subsection{The order of magnitude}

Let $F$ be a graph, $r\geq 2$.
Our aim is to determine the order of magnitude of the induced Berge Tur\'an number of $F$ as $n\to \infty$, or to reduce it to known problems.
Then in the next subsection we define a large class of 3-chromatic graphs $\GT$ which contains, e.g., all outerplanar graphs, and
 apply our results and methods to determine their induced Berge Tur\'an number more precisely.

\begin{thm}\label{mainbigr}
Let $r\geq 2$, and fix a graph $F$ such that $E(F)\neq \emptyset$. 
Then, as $n\to \infty$
   \[\ex_{r}(n, \IB F) = \Theta(\max_{2 \leq s \leq r}\{\ex(n,K_{s}, F) \}).\]
\end{thm}

This theorem shows that the order of magnitudes of the three functions in~\eqref{eqindF}
  behave differently as $r$ changes.
For small $r$, in the range $r \leq \chi(F) - 1$, all the three, $\ex_{r}(n, F)$, $\ex_{r}(n,\Ber F)$, and  $\ex_{r}(n,\IB F)$, are of order $\Theta(n^r)$ because
the balanced complete $(\chi(F) - 1)$-partite $r$-graph contains no Berge $F$ (so its 2-shadow, the $r$-partite Tur\'an graph is $r$-chromatic).

   If $r \geq |V(F)|$ then $\ex(n, K_r, F) = 0$ (since a $K_r$ contains a copy of $F$).
For general graphs $F$, the behavior of the three functions in the range $\chi(F) \leq r \leq |V(F)|-1$ is still unknown.
Determining the order of $\ex(n,K_r, F)$ for $r$ in this range would give an answer for the growth of $\ex_{r}(n,\IB F)$.

Concerning the Berge Tur\'an function Gerbner and Palmer~\cite{GP} showed that
    \[\ex_{r}(n,\Ber F) \leq \ex(n,F)\]
for $r\geq |V(F)|$.
So in this range  $\ex_{r}(n,\Ber F) =  O(n^2)$.
For the complete graphs the two sides have the same order: $\ex_{r}(n, \Ber K_r) = \Theta(n^2)$ if $r\geq 3$.
However this does not hold if $r$ is large compared to $|V(F)|$. %%% Recently
Gr\'osz, Methuku, and Tompkins~\cite{GMT} proved that for any non-bipartite $F$ and sufficiently large $r$, the order of $\ex_{r}(n,F)$ differs from that of $\ex(n,F)$: there exists some number $th(F)$ such that if $r \geq th(F)$ then $\ex_{r}(n,F)= o(n^2)$.

In contrast, the order of the induced Berge Tur\'an function $\ex_{r}(n,\IB F)$ is non-decreasing in $r$. Moreover, it is basically monotone. If $\bigcap E(F) =\emptyset$, i.e., $F$ is not a star, then we will see later by Lemma~\ref{lb} that
\begin{equation}\label{eq21}
  \left(1-\frac{r-1}{n}\right)\ex_{r-1}(n, \IB F) \leq \ex_{r}(n, \IB F).
\end{equation}

\subsection{Outerplanar graphs and more}

We define the class of $t$-vertex graphs $\GT^{(t)}$ by induction on $t$ as follows.
The class $\GT^{(2)}$ has only a single member, $K_2$.
For  $t>2$ one obtains each member $G$ of $\GT^{(t)}$ by taking a $G^{(t-1)}\in \GT^{(t-1)}$, taking an edge $xy\in G^{(t-1)}$, adding a new vertex $z\notin V(G^{(t-1)})$, and joining $z$ to $x$ and to $y$.
% , and then $G$ could be any non-empty subgraph of this.
Each $G\in \GT^{(t)}$ has exactly $t$ vertices and $2t-3$ edges.
Finally, let $\GT$ be the family of all non-empty subgraphs of the members of $\cup_{t\geq 2} \GT^{(t)}$.

Note that $\GT$ contains all outerplanar graphs, particuarly cycles, $C_t$, and forests.  Each $G\in \GT$ has chromatic number at most 3 and are obviously planar.

\begin{thm}\label{maincycle}
Let $r\geq 2$ be a positive integer.
Fix a graph $F\in \GT$. %  has $t$ vertices. % and at least two edges.
As $n\to \infty$ we have  $\ex_{r}(n, \IB F) = \Theta(\ex(n, F))$.
\end{thm}

This theorem reveals further gaps between $\ex_{r}(n, \Ber F)$ and $\ex_{r}(n, \IB F)$.
Gy\H{o}ri and Lemons~\cite{GL3,GL} proved that for $r\geq 3$ an $r$-uniform hypergraph avoiding a Berge cycle $C_{2t+1}$
has at most $O(\ex(n, C_{2t}))$ edges, which is known to be $O(n^{1 + (1/t)})$. 
On the other hand, in the same range, we have $\ex_{r}(n, \IB C_{2t+1})=\Theta(n^2)$.

Together, Theorems~\ref{mainbigr} and~\ref{maincycle} show that $\ex(n,C_t)$ has the same order as $\max_{2 \leq s \leq r}\{\ex(n,K_{s}, F) \}$. We obtain the following (known) corollary.
For any $r\geq 2$ and $t \geq 3$
\[\ex(n, K_r, C_t) = O(\ex(n, C_t)).\]
We also state the case of trees.

\begin{cor}\label{maintree}
Let $r \geq 2$ and $T$ be a forest with at least two edges. Then $\ex_{r}(n,\IB T) = \Theta(\ex(n,T))= \Theta(n)$.
\end{cor}

Finally, we get better bounds for stars, $F = K_{1,t-1}$. 
\begin{thm}\label{star}For any $r\geq 2$, $t \geq 3$, if $n = a(r+t-3) + b$ with $b \leq r+t-4$ then \[a{r+t-3 \choose r} + {b \choose r} \leq \ex_{r}(n, \IB K_{1,t-1}) \leq \frac{n}{r}{r + t - 3 \choose r-1} .\] In particular, if $n$ is divisible by $r+t-3$, the lower bound is $\frac{n}{r}{r+t-4 \choose r-1}$.
\end{thm}

\section{Constructions and proofs}

\subsection{Simple constructions and a monotonicity of the induced Berge Tur\'an function}

If $E(F)$ has a single edge then for $n\geq |V(F)|+r-2$ we have
   $\ex(n, F)=\ex(n,K_r, F) =\ex_{r}(n,\Ber F) = \ex_{r}(n, \IB F)=0$, so there is nothing to prove, all of our statements trivially hold.

In all other cases we have $\ex_{r}(n, \IB F)=\Omega(n)$ as one can see from the following constructions.
If $F$ has two non-disjoint edges then a matching of $r$-sets gives
  $\ex_{r}(n, \IB F)\geq \lfloor n/r \rfloor$.
If $F$ has two disjoint edges then the hypergraph consisting of $n-r+1$ sets sharing a common $(r-1)$-set yields $\ex_{r}(n, \IB F)\geq n-r+1$.

If $x\in V(F)$ is an isolated vertex then
$\ex_{r}(n, \IB F) = \ex_{r}(n, \IB (F\setminus \{ x\}))$ for all $n> (r-2)|E(F)|+|V(F)|$.
So we may delete isolated vertices and asymptotically get the same Tur\'an number.
From now on, we suppose that $F$ has no isolated vertex and $|E(F)|\geq 2$.

\begin{lem}\label{lb}Fix integers $r,t \geq 2$. If $F$ is a graph on $t$ vertices such that $F \neq K_{1, t-1}$ (and $e(F)\geq 2$ and $F$ has no islated vertex), then $\ex_{r}(n,\IB F) \geq \ex_{(r-1)}(n-1,\IB F)$. In particular, $\ex_{r}(n, \IB F) = \Omega(\ex(n,F))$.
\end{lem}

\begin{proof}Let $\mathcal H$ be an $(r-1)$-uniform hypergraph on $n-1$ vertices with $\ex_{r}(n-1, \IB F)$ edges and no induced Berge $F$. Construct an $r$-uniform hypergraph $\mathcal H'$ with $V(\mathcal H') = V(\mathcal H) \cup \{v\}$ such that the edges of $\mathcal H'$ are obtained by extending every edge of $\mathcal H$ to include the new vertex $v$. Suppose $\mathcal H'$ contains an induced Berge $F$. Since $\mathcal H$ was induced Berge $F$-free, $v$ must be a base vertex. Because $v$ is contained in every edge of $\mathcal H'$, there is a fixed vertex contained in every edge of $F$. I.e., $F = K_{1, t-1}$, a contradiction.

Inductively, we obtain $\ex_{2}(n-r+2, \IB F) \leq \ex_r(n, \IB F)$. But $\ex_2 (n-r+2, \IB F) = \ex(n-r+2, F) = \Theta(\ex(n, F))$.
\end{proof}

To show~\eqref{eq21} let $\cH$ be an induced Berge $F$-free $(r-1)$-uniform hypergraph on $n$ vertices, $|\cH|=\ex_{(r-1)}(n,\IB F)$.
For $x\in V:=V(\cH)$ let $\cH_x:=\{ e\in \cH: e\subset V\setminus \{x\}\}$.
Since each $\cH_x$ is also induced Berge $F$-free we get
\[  (n-r+1)\ex_{(r-1)}(n,\IB F) = (n-r+1)|\cH|=\sum_{x\in V}|\cH_x|\leq n\times
   \ex_{(r-1)}(n-1,\IB F).
\]
By Lemma~\ref{lb} the right hand side is at most $n \times
   \ex_{r}(n,\IB F)$. Rearranging yields~\eqref{eq21}. \qed

\subsection{The $\alpha$-core of a hypergraph}
Let $\cH$ be an $r$-partite, $r$-uniform hypergraph with parts $V(\cH) = V_1 \cup \ldots \cup V_r$. For some $1\leq s \leq r$ and edge $e \in \cH$, define $e[\overline{s}]$ to be the trace %%% projection XXX !!! ???
of $e$ onto all parts other than $V_s$. That is, $e[\overline{s}] = e \setminus V_s$. Let $\cH[\overline{s}] = \{e[\overline{s}]: e \in E(\cH)\}$.

\begin{thm}\label{core}
For positive integers $\alpha, r$, any $r$-uniform $r$-partite hypergraph $\cH$ contains edge-disjoint subhypergraphs $\cA$ and $\cB$ such that
\begin{enumerate}
\item[{\rm (a)}] For any $S \subseteq V(\cH)$, with $|S| = r-1$, either $\deg_{\cA}(S) = 0$ or $\deg_{\cA}(S) \geq \alpha$.
\item[{\rm (b)}] $|\cB| \geq \frac{|\cH \setminus \cA|}{\alpha - 1}$ and $|\cB| \leq \sum_{s=1}^r |\cB[\overline{s}]|$.
\end{enumerate}
\end{thm}

\begin{proof}
We build $\cA$ and $\cB$ inductively. Initially set $\cH_0 := \cH$, $\cB_0 := \{\emptyset\}$.

At step $i$, if there exists an $S \subseteq V(\cH_{i-1})$ with $|S| = r-1$ and $1\leq \deg_{\cH_{i-1}}(S) \leq \alpha -1 $, then let $E_S$ be the edges of $\cH_{i-1}$ containing $S$. Set $\cH_i = \cH_{i-1} \setminus E_S$. Pick any edge, say $B_i \in E_S$, and set $\cB_i = \cB_{i-1} \cup \{B_i\}$.

The process ends after $k$ steps when for every $S\subseteq V(\cH_k)$ with $|S|=r-1$, either $\deg_{\cH_k}(S) =0$ or $\deg_{\cH_k}(S) \geq \alpha$.  Let $\cA := \cH_k$ and $\cB := \cB_k = \{B_1 , \ldots , B_k\}$. Then $\cA$ satisfies $(a)$.

To see that $\cB$ satisfies $(b)$, at each step $i$ when we choose $B_i \in E_S$, $|E_S| \leq \alpha - 1$, so we obtain that $|\cB|$ is at least a $1/(\alpha - 1)$ portion of the deleted edges. Next, at each step, we associated with $B_i$ a distinct set $S_i$ of $r-1$ vertices. If $B_i$ and $B_j$ are associated with sets $S_i$ and $S_j$ respectively such that both sets are contained in $(V_1 \cup \ldots \cup V_r) \setminus V_s$, then in $\cB[\overline{s}]$, $B_i[\overline{s}]= S_i$ and $B_j[\overline{s}]=S_j$ are distinct. Hence $\sum_{s=1}^r \cB[\overline{s}] \geq |\{S_1, \ldots , S_k\}| = |\cB|$.
\end{proof}

Let any $\cA\subseteq \cH$ satisfying $(a)$ be called an {\bf $\alpha$-core} of $\cH$.

\begin{lem}\label{coreF}
Let $\alpha, r$ be positive integers, and let $F$ be a graph with $|V(F)|-1 %%%% !!!!  2
  \leq \alpha$. Let $\cH$ be an $r$-uniform, $r$-partite hypergraph with an $\alpha$-core $\cA$. If the 2-shadow $\partial_2\cA$ of $\cA$ contains a copy of $F$ then $\cA$
(and therefore $\cH$) contains an induced Berge $F$.
\end{lem}

\begin{proof}
We will find an induced Berge $F$ on the same base vertex set $V(F)$.
Let $xy$ be an edge in the copy of $F$, and let $e_{xy}$ be an edge of $\cA$ containing $\{x,y\}$ with minimum $|e_{x
y} \cap V(F)|$. Such an edge $e_{xy}$ exists by the definition of the 2-shadow. If $e_{xy}$ contains some vertex $z \in V(F) \setminus \{x,y\}$, then the $(r-1)$-set $e_{xy} \setminus \{z\}$ is contained in at least $\alpha-1$ other edges in $\cA$. Since there are $|V(F)| - 3 \leq \alpha - 2$ vertices in $V(F) \setminus \{x,y,z\}$, we may find some $z' \not\in V(F) - \{x,y,z\}$ such that $e_{xy} \setminus \{z\} \cup \{z'\} \in E(\cA)$, contradicting the choice of $e_{xy}$. Therefore $e_{xy} \cap V(F) =\{x,y\}$. We find such an edge of $\cA$ for each edge of $F$.
\end{proof}
If $\alpha \geq e(F) +|V(F)|$, then with the same method one can find an induced Berge $F$ in $\cA$ such that each pair of hyperedges $e_{xy}$ and $e_{uv}$ intersect only at $\{x,y\} \cap \{u,v\}$. This is called an $F$-\emph{expansion}.
But this observation does not seem to help our purposes here.

\begin{claim}\label{cl:34}
Suppose that $r\geq 3$ and $\cA$ contains an induced Berge $F$, where $|V(F)|\leq \alpha$ (and $E(F)\neq \emptyset$).
Define a new graph $F^+:=F^+_{xy}$  by adding a new vertex $z\notin V(F)$, taking an edge $xy\in E(F)$, and joining $z$ to $x$ and to $y$.
Then $\cA$ also contains an induced Berge $F^+$.\end{claim}

\begin{proof}
By Lemma~\ref{coreF}, there exists a hyperedge $e_{xy}\in \cA$ such that $e_{xy} \cap V(F) =\{x,y\}$.
Then for any $z'\in e_{xy} \setminus \{x,y\}$ we have that $xz'$ and $yz'\in \partial_2\cA$, so $F^+$ is a subgraph of $\partial_2\cA$.
Then Lemma~\ref{coreF} completes the Claim. \end{proof}

\begin{lem}\label{le:35}
Suppose that  $G\in \GT$ with $t=|V(T)|\geq 3$. Then $G\in \GT^{(t)}$.
\end{lem}

\begin{proof}
This statement seems to be evident, but still needs a proof.
By definition, there exists an $s\geq t$ such that $G\in \GT^{(s)}$.
Let $s=s(G)$ be the smallest such $s$. We will show by induction on $t$ that $s(G)=t$.
The base case $t=3$ is obvious. Suppose $t>3$ and that $G$ is a subgraph of $H\in \GT^{(s)}$, where the vertices of $H$ are $\{ v_1, \dots, v_s\}$ and each $v_i$ (with $i\geq 3$) has exactly two $H$-neighbors in $\{ v_1,\dots, v_{i-1}\}$. Moreover, these two neighbors (call them $v_{\alpha(i)}$ and $v_{\beta(i)}$) are joined by an edge in $H$.
Let $I\subseteq [s]$, $I:= \{ i_1, \dots, i_t\}$, $1\leq i_1 < \dots < i_t\leq s$, $V_I:=\{ v_i: i\in I\}$, and suppose that
$G$ is a spanning subgraph of $H[V_I]$. Since $s$ is minimal, we have $i_t=s$ and
 $N_H(v_s)=\{v_{\alpha(s)}, v_{\beta(s)}\}$.
$G':= H[V_I]\setminus \{ v_s \}$ has $t-1$ vertices, and it belongs to $\GT$.
By our induction hypothesis there exists a $H'\in \GT^{t-1}$ such that $G'$ is a subgraph of $H'$ on the same vertex set $V_I\setminus \{ v_s\}$.
If $\{v_{\alpha(s)}, v_{\beta(s)}\}\subseteq V(H')$ then by adjoining a new vertex $z'$ to $H'$ and connecting it to $v_{\alpha(s)}$ and $v_{\beta(s)}$ we obtain a $t$-vertex graph $H''$ from $\GT^{(t)}$ containing $G$.
If $|N_H(v_s)\cap V(H')|\leq 1$ then it is even simpler to find such a graph $H''$.
 \end{proof}

\subsection{Proofs of the upper bounds for induced Berge $F$ problems}

We prove a version of Theorem~\ref{mainbigr} with more precise bounds. For positive integers $a$ and $b$, $(a)_b = (a)(a-1) \cdots (a-b+1)$ denotes the falling factorial.

\begin{thm}\label{rpartite}
Let $t, r, n$ be positive integers, and let $F$ be any graph with $|V(F)| = t$. Let $\mathcal H$ be an $n$-vertex $r$-uniform hypergraph with no induced Berge $F$. If $\cH$ is $r$-partite, then \[e(\cH) \leq \sum_{i=2}^r (t-2)^{r-i}(r)_{r-i}\ex(n,K_i, F).\]
\end{thm}

\begin{proof}
We proceed by induction on $r$. The base case $r=2$ is trivial since an induced Berge $F$ is just a copy of $F$. Thus $\ex_2(n,\IB F) = \ex(n, K_2, F) = \ex(n,F)$. Now let $r \geq 3$. Let $\cA$ and $\cB$ be subhypergraphs of $\cH$ obtained from Theorem~\ref{core} with $\alpha = t - 1$. So we have \[|\cH| = |\cA| + |\cH \setminus \cA| \leq |\cA| + (t-2)\sum_{s=1}^r|B[\overline{s}]| \leq |\cA| + (t-2)(r)\ex_{r-1}(n,\IB F),\]
where the last inequality holds because each $B[\overline{s}]$ is $(r-1)$-uniform, $(r-1)$-partite and does not contain an induced Berge $F$.

By Lemma~\ref{coreF}, $\partial_2\cA$ contains no copy of $F$. Furthermore, since each edge in $\cA$ creates a $K_r$ in $\partial_2\cA$, $|\cA| \leq \ex(n, K_r, F)$. Applying the induction hypothesis, we obtain
\[|\cH| \leq \ex(n,K_r, F) +  (t-2)r\sum_{i=2}^{r-1} (t-2)^{r-1-i}(r-1)_{r-1-i}\ex(n,K_i, F)\]
and we are done.
\end{proof}

\begin{cor}
Let $t, r, n$ be positive integers, and let $F$ be any graph with $V(F) = t$. Then \[\max_{2 \leq s\leq r} \{\ex(n-(r-s), K_{s}, F)\} \leq \ex_{r}(n,\IB F) \leq \frac{r^r}{r!} \sum_{i=2}^r (t-2)^{r-i}(r)_{r-i}\ex(n,K_i, F).\]
In particular, $\ex_{r}(n,\IB  F) = \Theta(\max_{s \leq r}\{\ex(n, K_{s},F)\})$.
\end{cor}

\begin{proof}
The lower bound follows from Lemma~\ref{lb} and~\eqref{eqindF}. For the upper bound, we use the fact that any $r$-uniform hypergraph $\cH$ has an $r$-partite subhypergraph with at least $\frac{r!}{r^r}e(\cH)$ edges. Apply Theorem~\ref{rpartite} to any such subhypergraph.
\end{proof}

{\em Proof of Theorem~\ref{maincycle}}. The lower bound comes from Lemma~\ref{lb}.
For the upper bound, we proceed by induction on $r$.
First we show that if $\cH$ is $r$-partite with no induced Berge $F\in \GT$ then
\begin{equation}\label{eq3}
  |\cH| \leq  (t-2)^{r-2}\frac{r!}{2} \ex(n,F).
  \end{equation}
The base case $r=2$ is trivial, so let $r \geq 3$.
Let $\cA$ and $\cB$ be subhypergraphs of $\cH$ obtained from Theorem~\ref{core} with $\alpha = t-1$. Again we have
\begin{equation}\label{eqcycle}
   |\cH| \leq |\cA| + (t-2)\sum_{s=1}^r|B[\overline{s}]| \leq |\cA| + (t-2)(r) \ex_{r-1}(n,\IB F).
   \end{equation}
Observe that $\cA$ is empty. Indeed, if $\cA$ contains at least one edge, then the 2-shadow $\partial_2 \cA$ contains a $K_r$. So
Claim~\ref{cl:34} and Lemma~\ref{le:35} imply that $\partial_2 \cA$ contains a copy of $F$. Then we apply Lemma~\ref{coreF} to find an induced Berge $F$, a contradiction. Hence $|\cA| = 0$.
Applying induction hypothesis,~\eqref{eqcycle} yields~\eqref{eq3}.  % \[|\cH| \leq (t-2)^{(r-2)}r^{(r-2)}\ex(n,F).\]

Finally, if $\cH$ is not $r$-partite, then we apply the previous proof to an $r$-partite subgraph $\cH'$ of $\cH$ with at least $\frac{r!}{r^r}|\cH|$ edges to obtain $|\cH| \leq \frac{1}{2} r^{r}(t-2)^{r-2} \ex(n, F)$. \qed

{\em Proof of Theorem~\ref{star}.}
For the lower bound, let each component of $\cH$ be  a clique such that there are as many cliques of size $r+t-3$ as possible. If $n = a(r+t-3) + b$ where $0 \leq b < r+t-3$, then $|\cH| = a{r+t-3 \choose r} + {b \choose r}.$
Suppose $\cH$ contains an induced Berge $K_{1,t-1}$. Then its base vertices, say $\{v_1, \ldots, v_t\}$ must be contained in a single component of $\cH$. But each edge in a component contains at least 3 base vertices, a contradiction.

For the upper bound, let $\cH$ be an $n$-vertex, $r$-uniform hypergraph with no induced Berge $K_{1,t-1}$. We say that a set system $\{f_1, \ldots, f_s\}$ is {\em strongly representable} if for every $f_i \in \cF$, there exists a $v_i \in f_i$ such that $v_i \notin f_j$ for all $j \neq i$. F\"uredi and Tuza~\cite{FT} proved that if a set system $\cF$ with $|f| \leq r$ for all $f \in \cF$ does not contain a strongly representable subfamily of size $s$ then $|\cF| \leq {r + s-1 \choose r}$. For any vertex $v \in V(\cH)$, let $E_v:=\{e \setminus \{v\}: v \in e \in \cH\}$. The $(r-1)$-uniform set system $E_v$ cannot contain a strongly representable subfamily of size $t-1$, otherwise the corresponding edges in $\cH$ and their representative vertices would yield an induced Berge $K_{1,t-1}$ in $\cH$ with vertex $v$ as the center vertex. Therefore $\deg(v) \leq {(r-1)+(t-2) \choose r-1}$ so $|\cH| \leq \frac{n}{r} {r+t-3 \choose r-1}$.
\qed

% \bigskip

\end{document}